\documentclass[a4paper,11pt,reqno]{amsart}

\newcommand{\GFq}{\F_{q}}

\usepackage{amssymb}
\usepackage{latexsym}
\usepackage{amsmath}
\usepackage{graphicx}
\usepackage{amsthm}
\usepackage{empheq}
\usepackage{bm}
\usepackage[dvipsnames]{xcolor}
\usepackage{pagecolor}
\usepackage{subcaption}
\usepackage{tikz,pgfplots}
\pgfplotsset{compat=newest}

 \usepackage{array}

\usepackage[margin=3.2cm]{geometry}

\theoremstyle{definition}
\newtheorem{theorem}{Theorem}
\newtheorem{corollary}[theorem]{Corollary}
\newtheorem{proposition}[theorem]{Proposition}
\newtheorem{definition}[theorem]{Definition}
\newtheorem{example}[theorem]{Example}

\newtheorem{remark}[theorem]{Remark}
\newtheorem{lemma}[theorem]{Lemma}

\newcommand\qbin[3]{\left[\begin{matrix} #1 \\ #2 \end{matrix} \right]_{#3}}

\newcommand{\numberset}{\mathbb}
\newcommand{\N}{\numberset{N}}

\newcommand{\F}{\numberset{F}}

\newcommand{\fq}{\F_q}
\newcommand{\fqm}{\F_{q^m}}
\newcommand{\fqmn}{\F_{q^m}^n}

\newcommand{\Tr}{\textnormal{Tr}}

\newcommand{\mC}{\mathcal{C}}
\newcommand{\mS}{\mathcal{S}}

\newcommand{\mD}{\mathcal{D}}

\newcommand{\rk}{\textnormal{rk}}
\newcommand{\mB}{\mathcal{B}}

\newcommand{\mat}{\F_q^{n \times m}}
\renewcommand{\longrightarrow}{\to}

\newcommand{\GL}{\textnormal{GL}}
\newcommand{\SL}{\textnormal{SL}}

\definecolor{light-gray}{gray}{0.95}

\newcommand*{\myproofname}{Proof of the claim}

\usepackage{hyperref}
\hypersetup{
  colorlinks   = true, 
  urlcolor     = blue, 
  linkcolor    = blue, 
  citecolor   = red 
}

\begin{document}

\title[An Assmus-Mattson Theorem for Rank Metric Codes]{An Assmus-Mattson Theorem for Rank Metric Codes}
\author[Eimear Byrne]{Eimear Byrne}
\address{School of Mathematics and Statistics, University College Dublin, Belfield, Ireland}
\curraddr{}
\email{ebyrne@ucd.ie}
\thanks{}

\author[Alberto Ravagnani]{Alberto Ravagnani$^*$}
\address{School of Mathematics and Statistics, University College Dublin, Belfield, Ireland}
\curraddr{}
\email{alberto.ravagnani@ucd.ie}
\thanks{$^*$The author was partially supported by the Swiss National Science Foundation through grant n. P2NEP2\_168527 and by the the Marie Curie Research Grants
	Scheme, grant n. 740880.}



\maketitle

\thispagestyle{empty}

\begin{abstract}
A $t$-$(n,d,\lambda)$ design over $\GFq$, or a subspace design, 
	is a collection of $d$-dimensional subspaces of $\GFq^n$, called {\em blocks}, with the property that every     
	$t$-dimensional subspace of $\GFq^n$ is contained in the same number $\lambda$ of blocks. 
	A collection of matrices in $\GFq^{n \times m}$ is said to hold a $t$-design over $\GFq$ if the set of column spaces
	 of its elements forms the blocks of a subspace design. 
	 We use notions of puncturing and shortening of rank metric codes and the rank-metric MacWilliams identities to establish conditions under which the words of a given rank in a linear rank metric code hold a $t$-design over $\GFq$.
     We show that for $\fqm$-linear vector rank metric codes, the property of a code being MRD is equivalent to its minimal weight codewords holding trivial subspace designs and show that this characterization does not hold for $\GFq$-linear matrix MRD code that are not linear over $\fqm$. 
\end{abstract}

\section*{Introduction}\label{sec:intro}

The celebrated Assmus-Mattson Theorem \cite{AM} is among the best-known results in the theory of codes and designs. It has led to several constructions of $t$-designs \cite{AM,assmuskey,cameronvanlint,hupless,macws} and has been the subject of a number of generalizations \cite{bachoc,calddelsloane,del,moralestanaka,shinkumahelleseth}. 

The theorem specifies criteria under which the blocks of a combinatorial design are {\em held} by the words of a fixed weight in a linear code. It shows, in particular, that if the dual of a code has few non-zero weights in an interval, then the supports of codewords of fixed Hamming weight close enough to the minimum distance of the code form the blocks of a design.

The notion of a $t$-design over $\GFq$ (or a subspace design) has appeared in the literature since the 1970s \cite{cambcc}, and various constructions of such objects were given in the past \cite{Itoh,thomas,suzuki92}. 
A $t$-$(n,d,\lambda)$ design over $\GFq$ is a collection of $d$-dimensional subspaces of $\GFq^n$, called {\em blocks}, with the property that every $t$-dimensional subspace of $\GFq^n$ is contained in the same number $\lambda$ of blocks. 
There was a resurgence of interest in such designs during the last decade, due in part to the fact that some subspace designs are optimal as constant weight subspace codes. Subspace codes have been shown to have applications to error correction in network coding \cite{kk}. 
It was shown in \cite{flv} that non-trivial subspace designs exist over any finite field and for every value of $t$, as long as the parameters $n,d,\lambda$ are sufficiently large.
There are now several papers showing the existence of such objects for various parameter sets, most of which rely on assumptions of their automorphism groups \cite{braun05,braunetz,bkkljcta,braunkierwass,braunkierwass2}.

In this paper we give conditions under which a linear rank metric code yields a $t$-design over $\GFq$. 
Explicitly, we show that if $\mC$ and its dual $\mC^*$ are $w$ and $w^*$-{\em invariant}, respectively, and $\mC^*$ has at most $d-t$ non-zero weights in $\{1,...,n-t\}$, then the codewords of $\mC$ and $\mC^*$ of ranks $w$ and $w^*$, respectively, hold $t$-designs over $\F_q$. This result offers a new approach to the difficult problem of constructing new subspace designs and a motivation to study new classes of rank metric codes. 
We also consider existence problems applying recent results on the covering radius and the external distance of codes. These arguments apply to both the rank and the Hamming metric. 

In Section 2 we give preliminary results on rank metric codes and on subspace designs. In Section 3 we present our main result, the rank-metric analogue of the Assmus-Mattson Theorem.  
In Section 4 we restrict to the class of $\fqm$-$[n,k,d]$ linear codes, which are always $d$-invariant. We show that the {\em maximum rank distance} (MRD) codes, those with parameters $\fqm$-$[n,k,n-k+1]$, are characterized as precisely the codes whose minimal weight codewords hold {\em trivial} $t$-designs. We show, furthermore, that the codes that are linear only over $\fq$ do not satisfy this property. In Section 5 we discuss asymptotic existence of codes satisfying the hypothesis of the Assmus-Mattson Theorem, for both the Hamming and the rank metric. 



\section{Preliminaries}

\subsection{Matrix Rank Metric Codes}

We outline some preliminary results on rank metric codes. Throughout the paper, $q$ denotes a fixed prime power and $m,n$ are integers with
$m,n \ge 2$. We let $\F_q$ be the finite field with $q$ elements.

Given a pair of nonnegative integers $N$ and $M$, the $q$-\textbf{binomial} or \textbf{Gaussian coefficient} counts the number of $M$-dimensional subspaces of an $N$-dimensional subspace over $\fq$ and is given by:
$$\qbin{N}{M}{q}:=\prod_{i=0}^{M-1}\frac{q^N-q^i}{q^M-q^i}.$$

For a subspace $U$ of $\fq^n$ we define $U^\perp:=\{ v \in \fq^n : u \cdot v = 0\}$ to be the orthogonal space of $U$ with respect to the standard
inner product $u\cdot v =\sum_{i=1}^{n} u_i v_i$.

\begin{definition}
	The \textbf{rank distance} between a pair of matrices $X,Y \in \F_{q}^{n\times m}$ is defined to be the rank of their difference:
	$$d_{\rk}(X,Y):=\rk(X-Y).$$
\end{definition}

\begin{definition}
	A (linear) \textbf{matrix code} $\mC$ is an $\fq$-subspace of $\F_q^{n \times m}$. 
	The \textbf{minimum distance} of $\mC$ is $$d_\rk(\mC):= \min\{d_{rk}(X,Y) : X,Y \in \mC, \ X \neq Y\} = 
	\min\{\rk(X) : X \in \mC, \ X \neq 0\}.$$
	We say that $\mC$ is an $\fq$-$[n \times m,k,d]$ code if it has $\fq$-dimension $k$ and minimum rank distance $d$.
	The code $\mC^*$ denotes the \textbf{dual code} of $\mC$ with respect to the bilinear form on $\mat$ defined by
	$(X,Y) \mapsto  \langle X,Y \rangle := \Tr(XY^\top)$,
	that is, 
	$$ \mC^*:= \{ Y \in \fq^{n \times m}: \Tr(XY^\top)=0  \mbox{ for all }  X \in \mC\}.$$ 
\end{definition}

\begin{definition}
	Let $\mC$ be an $\fq$-$[n \times m, k,d]$ matrix code. 
	We define 
	$$W_i(\mC):=|\{X \in \mC : \rk(X)=i \}|,$$ which is the number of codewords of $\mC$ of rank $i$.
	The
	(\textbf{rank}) \textbf{weight distribution} of $\mC$ is the sequence $(W_i(\mC):  i \in \N)$. We say that an integer $i$ is a \textbf{non-zero weight} of $\mC$ if $W_i(\mC) \neq 0$.
\end{definition}

The weight distributions of a linear matrix code and its dual are related by the rank metric MacWilliams identities \cite{del}. We will use the following formulation from~\cite[Theorem 31]{alb1}.

\begin{theorem}\label{th:macwids}
	Let $\mC$ be an $\fq$-$[n \times m,k,d]$ matrix code. Then for each $\ell \in \{0,...,n\}$ we have:
	\begin{eqnarray}\label{eq:macw}
	\sum_{i=0}^{n-\ell} W_{i}(\mC^*)\qbin{n-i}{\ell}{q} & = & q^{m(n-\ell)k}
	\sum_{i=0}^{\ell} W_{i}(\mC)\qbin{n-i}{\ell-i}{q}.
	\end{eqnarray}
\end{theorem} 

The left-hand side of (\ref{eq:macw}) corresponds to multiplication of the weight distribution of $\mC$ by the upper triangular $q$-Pascal matrix
$$ M = \left(  \qbin{n-i}{\ell}{q} \right)_{0 \leq i,\ell \leq n}.$$
The $\ell \times \ell$ minors found in the first $\ell$ rows of $M$ can be expressed as
\begin{eqnarray}\label{eq:minors}
\det \left(  \qbin{r_i}{j-1}{q} \right)_{1 \leq i,j \leq \ell} = q^{\binom{\ell}{2}} \prod_{1\leq i < j \leq \ell} \frac{q^{r_j}-q^{r_i}}{q^j -q^i},
\end{eqnarray}
for some $r_1,...,r_\ell \in \{1,...,n\}$, (see, for example, \cite{dodgy}).

Clearly the determinant shown above is non-zero if the $r_j$ are all distinct and in particular if $r_j=n-i_j$, for distinct $i_1,...,i_\ell$. We will use this to establish invariance of the weight distribution of certain puncturings of a code whose dual code has few non-zero weights.

In the construction of a design over $\F_{q}$ from a set or space of matrices, we will identify the {\em blocks} of the design with the column spaces of the matrices.
This is the matrix analogue of the notion of the support of a vector.   

Given a matrix $X$ over $\fq$, we write $\sigma(X)$ to denote its column space, which we call the \textbf{support} of $X$. 
In fact this concept can be viewed in the more general lattice-theoretic setting of additive codes in groups \cite{alblatt}.

\begin{definition}
	Let $\mC$ be an $\fq$-$[n \times m, k,d]$ code. 
	Let $U$ be a subspace of $\fq^n$.
	We define
	\begin{equation*}
	\mC(U):= \{ X \in \mC : {\rm \sigma}(X) \le U \} \qquad \text{ and } \qquad \mC_=(U):= \{ X \in \mC : {\rm \sigma}(X) = U \}.
	\end{equation*} 
	Note that $\mC(U)$ is the subcode of $\mC$ consisting of all codewords of $\mC$ whose supports are contained in $U$.
	The set $\mC_=(U)$ comprises those codewords with support equal to $U$.  
\end{definition}

\begin{definition}
	Let $\mC$ be an $\fq$-$[n \times m, k,d]$ code.
	Let $U$ be a $u$-dimensional subspace of $\fq^n$.
	We say that $U$ is a \textbf{$u$-support} of $\mC$ if $\mC_=(U) \neq \emptyset$.
\end{definition}


\begin{definition}
	Let $\mC$ be an $\fq$-$[n \times m, k,d]$ code, let $A \in \GL_n(\fq)$ and let $S \subset [n]$.
	For each matrix $X \in \fq^{n\times m}$, we denote by $X_S$ the $|S| \times m$ matrix obtained from $X$ by deleting each $i$th row of $X$ for $i \notin S$.
	That is, $X_S$ is the projection of $X$ onto the rows indexed by $S$.
	We define the projection map 
	$$\pi_s: \fq^{n\times m} \to \fq^{(n-s) \times m}, \qquad X \mapsto X_{\{s+1,...,n\}}.$$
	We define $A \mC:=\{A X : X \in \mC \}$.
	Let $1 \leq s \leq n$. 
	The \textbf{punctured code} of $\mC$ with respect to $A$ and $s$ is
	$$\Pi(\mC,A,s):=\left\{ \pi_s(AX) = (AX)_{\{s+1,...,n\}} \ : \ X \in \mC \right\}.$$
	We define the \textbf{shortened code} of $\mC$ with respect to $A$ and $s$ by
	$$\Sigma(\mC,A,s):=\left\{ \pi_s(AX) \ : \ X \in \mC, \  (AX)_{\{1,...,s\}}=0 \right\}.$$	
\end{definition}	
In the case that $A$ is the identity matrix, the corresponding punctured code is found simply by deleting the first $s$ rows and the shortened code is found by selecting the codewords whose first $s$ rows are all-zero and then deleting these rows from each selected codeword.
We have the following duality result \cite{byrneravagnani}.
\begin{lemma} \label{lem:duality}
	Let $\mC$ be an $\fq$-$[n \times m, k,d]$ code and let $A \in \GL_n(\fq)$. Let $1 \leq s \leq n$ be an integer.
	Then
	\begin{equation*}
	\Pi(\mC,A,s)^* = \Sigma(\mC^*,(A^\top)^{-1},s).
	\end{equation*}
	In particular, $(A\mC)^*=(A^\top)^{-1} \mC^*$ for any $A\in \GL_n(\fq)$. 
\end{lemma}

We will apply Lemma \ref{lem:duality} a number of times throughout the paper.


\subsection{Designs over $\F_q$} 

We briefly recall the definition of $q$-design and the known constructions. In this paper, we only treat simple $q$-designs.

\begin{definition}
	Let $n\geq r \geq t$ be positive integers and let
	$\mD$ be a collection of $r$-dimensional subspaces of $\fq^n$. We say that $\mD$ forms (the \textbf{blocks} of) a \textbf{$t$-design} over $\fq$ if every $t$-dimensional subspace of $\fq^n$ is contained in the same number $\lambda$ of elements of $\mD$. In this case we say that $\mD$ is a $t$-$(n,r,\lambda)$ design over $\fq$. 
\end{definition}

Designs over $\fq$ are also known as {\em subspace designs} and as \textit{designs over finite fields}. A $t$-$(n,r,1)$ subspace design is called a $q$-Steiner system.
The interested reader is referred to the survey \cite{braunkierwass} and the references therein for an outline of the state of the art on designs over finite fields.

We have the following notion of a \textbf{dual subspace design} \cite{KP,suzuki}.

\begin{lemma}\label{lem:dualdesign}
	Let $n,r,t,\lambda$ be positive integers and let $\mD$ be a $t$-$(n,r,\lambda)$ design over $\fq$. Define
	$\mD^\perp:=\{ U^\perp : U \in \mD \}$. Then
	$\mD^\perp$ is a $$t- \left(n,n-r,\lambda\qbin{n-t}{r}{q}\qbin{n-t}{r-t}{q}^{-1} \right)$$ design over $\fq$.
\end{lemma} 

The {\em intersection numbers} associated with a subspace design were given in \cite{suzuki}. These fundamental design invariants often play an important role in establishing the non-existence of a design for a given set of parameters.

\begin{lemma}
	Let $n,r,t,\lambda$ be positive integers and let $\mD$ be a $t$-$(n,r,\lambda)$ design over $\fq$. 
	Let $I,J$ be $i,j$ dimensional subspaces of $\fq^n$ satisfying $i+j \leq t$ and $I \cap J = \{0\}$.
	Then the number
	$$\lambda_{i,j}:=|\{ U \in \mD : I \subseteq U, \ J \cap U = \{0\} \}| $$
	depends only on $i$ and $j$, and is given by the formula
	$$ \lambda_{i,j} = q^{j(r-i)}\lambda \qbin{n-i-j}{r-i}{q}\qbin{n-t}{r-t}{q}^{-1} .$$
\end{lemma}  

We list in Table \ref{taa} the known algebraic constructions of infinite families of 
{$t$-$(n,r,\lambda)$} subspace designs that do not rely on the existence of other subspace designs. 
This list is still rather short, although many specific numerical parameter sets are known to be realizable \cite{bkkljcta,braunkierwass,braunkierwass2}.
In most cases a computer search is combined with an assumed automorphism group, often in fact the normalizer of a Singer cycle group. One of the first such examples is the construction of a $3$-$(8,4,11)$ design over $\F_2$; see~\cite{braun05}.


The only known parameters of a $q$-Steiner system with $t \ge 2$ are given by 
$2$-$(13,3,1)$ over $\F_2$. Subspace designs for these parameters were found by an application of the Kramer-Mesner method and required significant computation \cite{braunetz}. The existence of the $q$-Fano plane, namely a $2$-$(7,3,1)$ design over $\fq$ is still an open problem. 

\medskip

{\small
	\begin{table}[h]
		\centering
		\begin{tabular}{|l|l|l|l|}
			\hline
			$t$-$(n,r,\lambda)$  & $\fq$ & Constraints & Refs               \\
			\hline \hline
			$2$-$(n,3,7)$ & $\F_2$ & $(n,6)=1, \ n \ge 7$&  \cite{thomas}\\
			\hline
			$2$-$\left(n,3,\qbin{3}{1}{q}\right)$ &$\fq$& $(n,6)=1, \ n \ge 7$ &  \cite{suzuki92}\\
			\hline
			$2$-$\left(n,r, {\displaystyle \frac{1}{2}}\qbin{n-2}{r-2}{q}\right)$ & $\F_3,\F_5$ &
			\begin{tabular}{ll}
				$n \ge 6$,
				$n \equiv 2 \mod 4$,\\
				$3 \le r \le n-3$,
				$r \equiv 3 \mod 4$
			\end{tabular} &  \cite{bkkljcta}\\
			\hline
		\end{tabular}
		\caption{Known primary constructions of infinite families of subspace designs.}
		\label{taa}
	\end{table}
}

Other realizable parameter sets can be obtained on condition of the existence of other subspace designs. 
In \cite{Itoh}, Itoh establishes the following construction.
\begin{theorem}
	Let $n \equiv 5 \mod 6(q-1)$. Suppose there exists a 
	$$\mbox{2-$\left(n,3, q^3\qbin{n-5}{1}{q}\right)$}$$ design over $\F_q$ that 
	that is invariant under the action of a Singer cycle of $\GL(n,q)$. Then for each $r \geq 3$, there exists a 
	$$\mbox{$2$-$\left(rn,3,q^3\qbin{n-5}{1}{q}\right)$}$$ design over $\F_q$ that is invariant under the action of $\SL(r,q^n)$.
\end{theorem}

As an example, one of Suzuki's $\mbox{$2$-$\left(n,3,\qbin{3}{1}{q}\right)$}$ designs over $\F_q$ yields a\\ $\mbox{$2$-$\left(n,3,q^3\qbin{n-5}{1}{q}\right)$}$ subspace design, whose $\qbin{n}{3}{q}-\qbin{3}{1}{q}$ blocks form
the complement of the set of blocks of the original design. Moreover, both of these are invariant under the action of a Singer cycle in $\GL(n,q)$.
Hence the infinite family of designs constructed in \cite{suzuki92} yields a new infinite family via Itoh's construction. 
Some further infinite families of $2$-$(rn,3,\lambda)$ designs over $\F_q$ have been constructed by Itoh's result, based on sporadic examples of $2$-$(n,3,\lambda)$ Singer cycle invariant subspace designs found by computer search; see \cite[Table 5]{braunkierwass}. 

Finally, in \cite{bkkljcta}, the authors show how to construct new {\em large sets} of designs from existing ones. An $LS_q[N](t,r,n)$ large set is a partition of the set of $\qbin{n}{r}{q}$ $r$-dimensional subspaces of $\F_q^n$
into $N$ disjoint $t$-$\left(n,r,\qbin{n-t}{r-t}{q}N^{-1}\right)$ designs. It is known, for example, that $LS_q[N](2,r,n)$ large sets exist for $q=3,5$, $n\geq 6$ such that $n\equiv 2 \mod 4, $ and $3\leq r \leq n-3$ satisfying $r \equiv 3 \mod 4$. In \cite[Table 7]{braunkierwass} the authors list 5 large set parameters that are known to be {\em realizable}, and many more that are {\em admissable} but for which existence is not yet established.
The interested reader is referred to \cite{bkkljcta} for further details.

\section{An Assmus-Mattson Theorem for the Rank Metric}

The main result of this section is Theorem \ref{th:AM}, which gives criteria under which the $w$-supports of an $\fq$-$[n \times m,k,d]$
form a design over $\fq$. This is a rank-metric analogue of the Assmus-Mattson theorem, connecting codes and subspace designs. 

Throughout this section, if $\mC$ is an $\fq$-$[n \times m, k,d]$ code and $U \le \F_q^n$ is a subspace, then we let
$$\mC_{\ge}(U,d):= \{X \in \mC : \rk (X) = d, \ \sigma(X) \ge U\} \subseteq \mC.$$
We also define 
$$\pi_{t}(\mC)_{\ell}:=\{X \in \pi_t(\mC) \ : \  \rk(X)=\ell\} \qquad \text{and} \qquad \Pi(\mC,A,t)_{\ell} := \pi_t(A\mC)_\ell.$$

We start with a series of preliminary results. In the sequel, we let $e_i$ denote the unit or standard basis vectors.

\begin{lemma} \label{lem:canspaces}
	Let $\mC$ be an $\fq$-$[n \times m, k,d]$ code. Let
	$0 \le t  <d$ be an integer, and let $E:= \langle e_i \ : \  1 \le i \le t \rangle \le \F_q^n$.
	There is a one-to-one correspondence
	between the words of rank $d$ in $\mC$ whose supports contain $E$ and the elements of the punctured code $\pi_t(\mC)$ of rank $d-t$. More precisely, the restriction of $\pi_t$ to 
	$\mC_{\ge}(E,d)$ is a bijection:
	$$\pi_t:\mC_{\ge}(E,d) \to \pi_{t}(\mC)_{d-t}.$$
\end{lemma}
\begin{proof}
	First observe that $\pi_t$ is an $\fq$-linear homomorphism of $\mC$ onto $\pi_t(\mC)$ and moreover is an injection, since for any non-zero $X \in \mC$,
	the rank of $\pi_t(X)$ is at least $d-t >0$.	
	Let $X \in \mC_{\ge}(E,d)$. Since $E \le \sigma(X)$, there exists an invertible matrix $G \in \GL_m(\F_q)$ such that
	\begin{equation*} \label{dec1}
	X= \begin{pmatrix} I_t & 0 \\ 0 & Z \end{pmatrix}  G,
	\end{equation*}
	for some $(n-t) \times (m-t)$ matrix $Z$ of rank $d-t$. Therefore, $\pi_t(X)= \begin{pmatrix} 0& Z \end{pmatrix} G$
	and $d=\rk(X)=t+\rk(Z)$.
	In particular, $$\rk(\pi_t(X)) =\rk(Z)=d-t.$$ 
	This shows that $\pi_t$ is a well-defined, injective map from $\mC_{\ge}(E,d)$ to
	$$\pi_{t}(\mC)_{d-t} =\{X \in \pi_t(\mC) \ : \  \rk(X)=d-t\} .$$
	
	Let $Y \in \pi_t(\mC)$ such that $\rk(Y)=d-t$. By definition, there exists $X \in \mC$ such that $\pi_t(X)=Y$. 
	Since $d-t>0$, we have $X \neq 0$. Therefore
	$$d \le \rk(X) \le \rk(\pi_t(X))+t=\rk(Y)+t=(d-t)+t=d.$$
	It follows that $\rk(X)=d$.
	We will show that $E \le \sigma(X)$. Towards a contradiction, suppose that 
	$\sigma(X)$ does not contain $E$. Define the $n \times (m+t)$ matrix
	$$\overline{X}:= \begin{pmatrix} X & e_1^\top & \cdots & e_t^\top\end{pmatrix}.$$
	We then have $\rk(\overline{X}) \ge \rk(X)+1$, as the span of
	$e_1^\top,...,e_t^\top$ is not contained in the span of the columns of $X$. Therefore 
	\begin{equation*} \label{contr}
	\rk(\pi_t(\overline{X})) \ge \rk(\overline{X})-t \ge \rk(X)+1-t=d+1-t.
	\end{equation*}
	On the other hand, by construction of $\overline{X}$ we have
	$$\rk(\pi_t(\overline{X})) = \rk(\pi_t(X))=\rk(Y)=d-t,$$
	yielding a contradiction. We deduce that $\pi_t$ is a bijection onto $\pi_t(\mC)_{d-t}$. 
\end{proof}

\begin{corollary} \label{pr:1}
	Let $\mC$ be an $\fq$-$[n \times m, k,d]$ code. Let
	$0 \le t  <d$ be an integer, and let $T\le \F_q^n$ be a subspace of dimension $t$. There exists
	$A \in \GL_n(\F_q)$ such that the
	words of weight $d$ in $\mC$ whose support contains $T$ are in one-to-one correspondence with the words of weight $d-t$ in $\Pi(\mC,A,t)$.
\end{corollary}

\begin{proof}
	Let $f:\F_q^n \to \F_q^n$ be an $\F_q$-isomorphism with the property that $f(T)=E= \langle e_i \ : \  1 \le i \le t \rangle$. Let $A \in \GL_n(\F_q)$ be a representation of $f$. Define $\overline{\mC}:=A\mC$.
	Multiplication by $A$ from the left gives a bijection from $\mC_{\ge}(T,d)$ onto $\overline{\mC}_{\ge}(E,d)$. 
	The result now follows since, 
	by Lemma \ref{lem:canspaces}, we have 
	$$|\overline{\mC}_{\ge}(E,d)|=| \pi_t(\overline{\mC})_{d-t}|=W_{d-t}(\pi_t(A\mC))=W_{d-t}(\Pi(\mC,A,t)).$$
	We have shown that for an arbitrary $t$-dimensional space $T$, there is a bijection between $\mC_\geq(T,d)$ and $\Pi(\mC,A,t)_{d-t}$, where $A$ is an invertible matrix mapping a basis of~$T$ to $\{e_1,...,e_t\}$.
\end{proof}

\begin{proposition}\label{prop:punwd}
	Let $\mC$ be an $\F_{q}$-$[n\times m,k,d]$. Let $t <d$ and let $W_i(\mC^*)$ be non-zero for at most $d-t$ values of $i$ in $\{1,...,n-t\}$.
	Let $A$ be an invertible matrix in $\fq^{n \times n}$.
	Then the weight distribution of $\Pi(\mC,A,t)$ is determined and is independent of~$A$.
\end{proposition}

\begin{proof}
	Let $\{i_1,...,i_r\}$ be the set of non-zero weights of $\mC^*$  in $\{1,...,n-t\}$, for some $r \leq d-t$.	
	Let $E=\langle e_{t+1},...,e_{n} \rangle$.
	The non-zero weights in $\{1,...,n-t\}$ of 
	$\mC^*(E) \subseteq \mC^*$ 
	form a subset of $\{i_1,...,i_r\}$ and since every element of $\mC^*(E)$ has rank at most $n-t=\dim E$, all of the weights of $\mC^*(E)$ lie in $\{i_1,...,i_r\}$.
	Moreover, since the first $t$ rows of each element of $\mC^*(E)$ are all zero, the map  
	$$\Phi:\mC^*(E) \longrightarrow \Sigma(\mC^*,I,t) \ : \  X \mapsto \pi_t(X)$$ 
	is a rank preserving bijection. As a consequence, the 
	non-zero weights of the code $\Sigma(\mC^*,I,t) = \Pi(\mC,I,t)^*$ form a subset of $\{i_1,...,i_r\}$. Clearly, the non-zero weights of $\mC^*$ are precisely the non-zero weights of $(A\mC)^*$. Therefore, we may apply the same argument with
	$A\mC$ in the place of $\mC$ to conclude that the non-zero weights of $$\Pi(\mC,A,t)^*=\Sigma(\mC^*,(A^\top)^{-1},t)=\Sigma((A\mC)^*,I,t) $$ form a subset of $\{i_1,...,i_r\}$, the weights of $\mC^*$ in $\{1,...,n-t\}$.
	
	Now observe that $W_i(\Pi(\mC,A,t))=0$ for $i=1,...,d-t-1$, since the minimum distance of $\Pi(\mC,A,t)$ is at least $d-t$. 
	In particular, at least $d-t$ of the values of the weight distribution of $\Pi(\mC,A,t)$ are known.

	The MacWilliams duality theorem for rank metric codes \cite{del,alb1} then yields a system of $d-t$ equations in at most $d-t$ unknowns.
	In fact, using Theorem \ref{th:macwids}, for $\ell = 0,...,d-t-1$ these equations are explicitly given by
	\begin{eqnarray*}
		\sum_{i=0}^{n-\ell} W_{i}(\Pi(\mC,A,t)^*)\qbin{n-i}{\ell}{q} & = & q^{m(n-\ell)-k}
		\sum_{i=0}^{\ell} W_{i}(\Pi(\mC,A,t))\qbin{n-i}{\ell-i}{q},
	\end{eqnarray*}
	which is equivalent to the system,
	\begin{eqnarray*}
		\sum_{j=1}^{r} W_{i_j}(\Pi(\mC,A,t)^*)\qbin{n-i_j}{\ell}{q} & = & 
		(q^{m(n-\ell)-k}-1)\qbin{n}{\ell}{q} .
	\end{eqnarray*} 
	From (\ref{eq:minors}), this system can be solved
	for unique $W_{i_j}(\Pi(\mC,A,t)^*), 1 \leq j \leq r$. Furthermore, this solution is independent of $A$, so 
	the weight enumerator of $\Pi(\mC,A,t)^*$, and hence $\Pi(\mC,A,t)$, is independent of $A$. 		
\end{proof}

We require one more definition before proving the main result of this section.

\begin{definition} \label{def:u-inv}
	Let $\mC$ be an $\F_{q}$-$[n\times m,k,d]$ code and let $w$ be an integer satisfying $d \leq u \leq n$. 
	We say that $\mC$ is \textbf{$u$-invariant} if $|\mC_=(U)|$ depends only on $u$ for each $u$-support $U$.
	If $\mC$ is $u$-invariant we define $\mu(\mC,u):=|\mC_=(U)|$, where $U$ is any $u$-support.
\end{definition}

\begin{lemma}\label{lem:inv}
	Let $\mC$ be an $\F_{q}$-$[n\times m,k,d]$ $w$-invariant code for some $d\le w \le n$. Let
	$A \in \GL_n(\fq)$. Then $A\mC$ is also $w$-invariant and
	$\mu(\mC,w) = \mu(A\mC,w)$.
\end{lemma}

\begin{proof}
	The map $f: \mC \longrightarrow A\mC : X \mapsto AX$ is a rank preserving bijection.
	Let $U \leq \fq^n$ be a $w$-dimensional space and let $X \in \mC$. Then $U = \sigma (X)=\{ Xv^\top:v \in \fq^n\}$ 
	if and only if $f(U) = \{AXv^\top:v \in \fq^n \} = \sigma(AX)$. The result follows.
\end{proof}

The following is a rank metric analogue of the Assmus-Mattson Theorem \cite[Theorem 4.2]{AM}. As in the classical case, we show that if the dual code of a $w$-invariant
rank metric code has few weights, then the $w$-supports of $\mC$ form the blocks of a subspace design.

\begin{theorem} \label{th:AM}
	Let $\mC$ be an $\fq$-$[n \times m, k,d]$ code. Let $1 \le t<d$ be an integer, and assume that
	$$|\{1 \le i \le n-t \ : \ W_i(\mC^*) \neq 0\}| \le d-t.$$
	Denote by $d^*$ the minimum distance of $\mC^*$, and let $w,w^*$ be integers satisfying $d \leq w \leq n$ and $d^* \leq w^* \leq n$. Suppose that $\mC$ (resp. $\mC^*$) is $u$-invariant for each $u\leq w$ (resp. $u\leq w^*$).
	Then the $u$-supports of $\mC$ (resp. $\mC^*$) form the blocks of a $t$-design over $\fq$ for each $d\leq u \leq w$ (resp. $d^* \leq u \leq \min \{w^*,n-t\}$).
\end{theorem}

\begin{proof}
	
	Throughout this proof, let $T$ be a fixed $t$-dimensional subspace of $\fq^n$, let $E=\langle e_1,...,e_t \rangle$ and let $A$ be a matrix representation of the $\fq$-isomorphism that maps $T$ to $E$. From Lemma \ref{lem:canspaces} and the proof of 
	Corollary \ref{pr:1}, the words of weight $d$ in $\mC$ whose supports contain $T$ are in one-to-one correspondence with the words of weight $d-t$ in the punctured code $\Pi(\mC,A,t)$. That is, there is a bijection from $\mC_\geq(T,d)$ to $\Pi(\mC,A,t)_{d-t}$.
	
	We will count in two ways the elements of the set:
	$$P=\{(U,X) \ : \  U \le \F_q^n, \ \dim(U)=d, \ T \le  U , \ X \in \mC, \  \sigma(X)=U\}.$$
	On one hand, since $\mC$ is $d$-invariant by hypothesis, we have:
	$$|P|= \sum_{T \leq U, \ \dim U=d } |\mC_=(U)| \ = \  \mu(\mC,d)  |\{U \le \F_q^n \ : \  U \mbox{ is a $d$-support of $\mC$}, T \le  U \}|.$$
	On the other hand, since each $X \in \mC$ has a uniquely determined support $U=\sigma(X)$ we have:
	$$|P|= |\{X \in \mC \ : \  \rk(X)=d, \ T \le \sigma(X) \}| =|\mC_\geq(T,d)|.$$ 
	By Corollary \ref{pr:1}, there exists $A \in \GL_n(\F_q)$ such that
	$|P| = W_{d-t}(\Pi(\mC,A,t))$. By Proposition \ref{prop:punwd}, the value of $W_{d-t}(\Pi(\mC,A,t))$ only depends on $k$, $d$ and $t$.
	Now,
	$$W_{d-t}(\Pi(\mC,A,t)) = \sum_{\substack{U \ge T \\ \dim U=d}} |\mC_=(U)|  \ = \ \mu(\mC,d)|  \{ U \ : \ T\leq U \text{ is a }d\text{-support of }\mC\}|,$$
	and hence the number of $d$-supports of $\mC$ that contain $T$ is independent of the choice of $T$ of dimension $t$.
	We deduce that the $d$-supports of $\mC$ form the blocks of a $t$-design over $\fq$.	
	
	We now apply an inductive argument to show that the $u$-supports of $\mC$ form the blocks of a $t$-design for each $d\leq u \leq w$. Suppose the result holds 
	for $d\leq u<w$.  
	Let $\lambda_{k,\ell}(r)$ denote the number of $r$-supports $U$ of $\mC$ satisfying $U \cap L = \{0\}$ and $K \subseteq U$ for a fixed $k$-dimensional space $K$ and an $\ell$-dimensional space $L$. In other words, the $\lambda_{k,\ell}(r)$ are the intersection numbers of the designs held by the words of rank $r$ in $\mC$.
	We will count the number of words of $\Pi(\mC,A,t)$ of weight $w-t$. Let $\lambda_T(w)$ denote the number of $w$-supports of $\mC$ that contain $T$.
	
	Let $Y \in \Pi(\mC,A,t)$ have rank $w-t$. Since $\pi_t$ is an injection on $\mC$, $Y=\pi_t(AX)$ for a unique matrix $X \in \mC$ of rank at most $w$. 
	If $\dim(\sigma(X) \cap T)=t-\ell$, then there exists an invertible matrix $B$ such that 
	$$ AXB =  \begin{pmatrix}
	M_1 &| & M_2 \\
	0   &| & M_3
	\end{pmatrix},$$
	for some matrices $M_i$, where $M_1$ is a $t \times t$ matrix of rank $t-\ell$ and $M_3$ is a matrix of rank $w-t$, satisfying 
	$$\sigma\left(\left[ \begin{array}{c}
	M_1 \\
	0
	\end{array}\right] \right)=\sigma(AX) \cap E \text{ and } \sigma(M_3)=\sigma(Y),$$ in which case $\rk(X) = w-\ell$.
	Therefore every word of rank $w-t$ in $\Pi(\mC,A,t)$ corresponds to a unique word in $\mC$ of weight $w -\ell$ whose support meets $T$ in a space of dimension $t-\ell$.
	
	Now let $V$ be a $(t-\ell)$-dimensional subspace of $T$ 
	and let $L$ be an $\ell$-dimensional subspace of $T$ satisfying $T = V + L$. Clearly, $L \cap \sigma(X)=\{0\}$ for any $X \in \mC$ with
	$\sigma(X) \cap T = V$.
	Then $$|\{ X \in \mC \ : \ \rk(\pi_t(X)) = w-\ell, \ L \cap \sigma(X) =\{0\}, \ V \subset \sigma(X) \}|=\mu(\mC,w-\ell)\lambda_{t-\ell,\ell}(w-\ell).$$ 
	It follows that  
	$$W_{w-t}(\Pi(\mC,A,t))= \lambda_T(w)\mu(\mC,w) + \sum_{\ell=1}^{w-t} \qbin{t}{\ell}{q} \lambda_{t-\ell,\ell}(w-\ell) \mu(\mC,w-\ell).$$
	Then $\lambda_T(w)$ is independent of $T$ and hence the $w$-supports of $\mC$ are the blocks of a $t$-design over $\fq$.
	
	Now consider the words of weight $d^*$ in $\mC^*$. We claim the $d^*$-supports of $\mC^*$ are the blocks of a $t$-design over $\fq$.
	For each $d^*$-dimensional subspace $U \le \fq^n$ define
	\begin{eqnarray*}
		\mS(\mC^*,U) &:=& 
		\{\sigma(X) \ : \  X \in \mC^*(U), \ \rk(X)=d^* \} \\ &=& \{\sigma(X) \ : X \in \mC^*,  \ \sigma(X) \le U, \ \rk(X)=d^*\}.
	\end{eqnarray*}
	Furthermore, define $\mS^\perp(\mC^*,U):=\{ V^\perp  \ : \ V \in \mS(\mC^*,U) \}$.
	Note that $U^\perp$ is contained in every element of $\mS^\perp(\mC^*,U)$ as
	$\sigma(X) \le U$ if and only if $U^\perp \le \sigma(X)^\perp$ for each $X \in \mC^*$.
	
	Recall $E=\langle e_1,...,e_t \rangle$ and so $E^\perp = \langle e_{t+1},...,e_n \rangle$. 
	We now compute $|\mS^\perp(\mC^*,E^\perp)|$, which is equal to $|\mS(\mC^*,E^\perp)|$.

	The map
	$$\Phi : \mC^*(E^\perp) \longrightarrow \Sigma(\mC^*,I,t) \ : \ X \mapsto \pi_t(X),$$
	is a rank preserving bijection, as any $n \times m$ matrix with column space in $E^\perp$ has its first $t$ rows all-zeroes. 
	In particular, 
	$$W_{d^*}(\mC^*(E^\perp)) = W_{d^*}(\Sigma(\mC^*,I,t))= W_{d^*}(\Pi(\mC,I,t)^*)=W_{d^*}(\pi_t(\mC)^*). $$
	By assumption, $\mC^*$ and hence $\mC^*(E^\perp)$ is $d^*$-invariant, so that
	\begin{equation}\label{eq:supp}
	|\mS(\mC^*,E^\perp)|  \mu(\mC^*,d^*)  =  W_{d^*}(\mC^*(E^\perp)) = W_{d^*}(\Pi(\mC,I,t)^*).
	\end{equation}
	
	We can now compute $|\mS^\perp(\mC^*,T^\perp)|$. Since $\sigma(X) \leq E^\perp$ if and only if $\sigma(AX) \leq T^\perp$, we have that 
	\begin{eqnarray*}
		\mS(\mC^*,E^\perp) & =& 
		\{ \sigma(X) \ : \ \sigma(AX) \le T^\perp, \ \rk(X)=d^*, \ X \in \mC^* \}\\
		& = &  
		\{ \sigma(A^{-1} X) \ : \ \sigma(X) \le T^\perp, \ \rk(X)=d^* , \ X \in A(\mC^*) \}
	\end{eqnarray*}
	
	Substituting $A^\top\mC$ for $\mC$ in the above and the fact that 
	$(B\mC)^* = (B^\top)^{-1} \mC^*$ for any $B \in$ GL$_n(\fq)$ yields
	$$	    \mS((A^\top\mC)^*,E^\perp)  = 
	\{ \sigma(A^{-1}X) \ : \ \sigma(X) \le T^\perp,  \ \rk(X)=d^* , \ X \in A(A^\top\mC)^*=\mC^* \}.
	$$
	It follows that $|\mS((A^\top\mC)^*,E^\perp)| = |\mS(\mC^*,T^\perp)|$. 
	
	Since $\mu(A^\top \mC,d^*)=\mu(\mC,d^*)$ from Lemma \ref{lem:inv}, again substituting $A^\top\mC$ for $\mC$ in~(\ref{eq:supp}), we get that
	$$|\mS^\perp(\mC^*,T^\perp)| = |\mS((A^\top\mC)^*,E^\perp)|= \frac{W_{d^*}(\Pi(\mC^*,A^\top,t))}{\mu(\mC^*,d^*)}. $$ 
	Since the weight distribution of $\Pi(\mC^*,A^\top,t))$ is independent of $A$, it follows that the number $|\mS^\perp(\mC^*,T^\perp)| $ is independent
	of $T$ of dimension $t$. 
	
	Let $\mB$ be the set of $d^*$-supports of $\mC^*$ and let $\mB^\perp:=\{U^\perp  : U \in \mB\}$.
	The argument just given shows that every $t$-dimensional subspace $T$ is contained in the same number of elements of $\mB^\perp$, which is therefore a $t$-design over $\fq$. Therefore $\mB$ forms the blocks of $t$-design, which is the {\em dual} subspace design (as described in Lemma \ref{lem:dualdesign}) of $\mB^\perp$.
	
	Using the same inductive argument as before, along with the fact that the weight distribution of $\Pi(\mC,A,t)^*$ is determined, we see that the words of weight $i$ for each $t \le i \leq \min \{n-t,w^*\}$ form the blocks of a $t$-design.
\end{proof}


\section{Vector Rank Metric Codes and $q$-Designs}

Rank metric codes that arise from a vector code $C$ that is linear over an extension field $\F_{q^m}$ are natural candidates as 
$u$-invariant codes (see Definition~\ref{def:u-inv}). As the reader will see, such a code $C$ always induces a $d$-invariant matrix rank metric code, where $d$ denotes the minimum distance of $C$.

\begin{definition}
	Let $\Gamma$ be a basis of $\F_{q^m}$ over $\fq$. 
	For each $x \in \F_{q^m}^n$, we write $\Gamma(x)$ to denote the $n \times m$ matrix over $\fq$ whose $i$th row is the coordinate vector of
	the $i$th coefficient of $x$ with respect to the basis $\Gamma$. In other words, $\Gamma(x)$ is defined via
	$$x_i=\sum_{j=1}^m \Gamma(x)_{ij} \gamma_j, \qquad 1 \le i \le n.$$
	The \textbf{rank} of $x$ is the rank of the matrix $\Gamma(x)$. Note that the rank of $x$ is well-defined, as it does not depend on the choice of the basis $\Gamma$.
\end{definition}

The spaces $\fq^{n\times m}$ and $\F_{q^m}^n$ are isomorphic as $\fq$-vector spaces with respect to the map $x \mapsto \Gamma(x)$, which, by abuse of notation, we denote by $\Gamma$ as well. The image
of an $\fqm$-linear subspace $C \le \fqmn$ under $\Gamma$ is therefore the matrix code
$\Gamma(C)= \{ \Gamma(x): x \in \mC \}$.
We have $$\dim_{\F_q}(\Gamma(C)) = m  \dim_{\fqm}(C).$$

\begin{definition}
	A (linear rank-metric) \textbf{vector code} $C$ is an $\fqm$-subspace of $\fqmn$ with $C \neq \{0\}$ and $C \neq \fqmn$.
	The \textbf{minimum distance} of $C$ is the minimum distance of $\Gamma(C)$, where $\Gamma$ is any $\F_q$-basis of $\fqm$. 
	We say that $C$ is an $\F_{q^m}$-$[n,k,d]$ code if it has $\fqm$-dimension $k$ and $\Gamma(C)$ has minimum rank distance $d$.
	The code $C^\perp$ denotes the \textbf{dual code} of $C$ with respect to the standard inner product of $\fqmn$. 
\end{definition}

We have the following duality result (see \cite[Theorem 21]{alb1}).

\begin{lemma}\label{dualalb}
	Let $C \le \fqmn$ be a vector code.
	Let $\Gamma$ be an $\F_q$-basis of $\fqm$, and let $\overline{\Gamma}$ denote the
	trace-dual basis of $\Gamma$. We have $$\Gamma(C)^* = \overline{\Gamma}(C^\perp).$$
\end{lemma}

We also recall the Singleton-type bound for $\F_q$-linear matrix codes \cite{del}.

\begin{proposition}\label{stbound}
	Let $\mC$ be an $\fq$-$[n \times m, k,d]$ matrix code. Then the dimension of $\mC$ is upper bounded as follows:
	$$k \le \max\{n,m\}  (\min\{n,m\}-d+1).$$
\end{proposition}

Codes that meet the bound of Proposition \ref{stbound} are called \textbf{maximum rank distance} (MRD) codes.

\begin{definition}
	An $\F_{q^m}$-$[n,k,d]$ vector code is called \textbf{MRD} if one of the following equivalent properties hold:
	\begin{itemize}
		\item $\Gamma(C)$ is MRD for some $\fq$-basis $\Gamma$ of $\fqm$.
		\item $\Gamma(C)$ is MRD for all $\fq$-basis $\Gamma$ of $\fqm$.
	\end{itemize}
\end{definition}

It is known \cite{del} that $\fqm$-linear vector MRD codes exist for all choices of the admissible parameters. In particular, $\F_q$-linear MRD matrix codes exist for all choices of the admissible parameters.

An $\fq$-$[n \times m, k,d]$ matrix code $\mC \le \mat$ is called \textbf{dually QMRD} if $$d_\rk(\mC)+d_\rk(\mC^\perp)=\min\{m,n\}+1.$$ These codes can be regarded as the best alternative to MRD codes for dimensions that are not a multiple of $\max\{m,n\}$. They share with MRD codes some important rigidity properties. We refer the interested reader to \cite{cruzrav} for further details.
As with MRD codes, an $\F_{q^m}$-$[n,k,d]$ vector code $C$ is dually QMRD if
$\Gamma(C)$ is dually QMRD for some $\fq$-basis $\Gamma$ of $\fqm$.

\begin{lemma} \label{lem:bnd}
	Let $C$ be an $\F_{q^m}$-$[n,k,d]$ code, and let $U \le \F_q^n$ be a subspace of dimension $u \ge d+1$. For all $\F_q$-bases $\Gamma$
	of $\fqm$ we have 
	$$\dim_{\F_q} (\Gamma(C)(U)) \leq m(u - d +1).$$
\end{lemma}
\begin{proof}
	Let $\mC:=\Gamma(C)$, let $E= \langle e_1,...,e_u \rangle$ and let $f:\F_q^n \to \F_q^n$ be an $\F_q$-isomorphism such that $f(U) =E$.
	Then $\mC(U) = (A\mC)(E)$.
	Clearly, $\mC$ and $A\mC$ have the same minimum distance. Therefore we can view $(A\mC)(E)$ as a code in $\F_q^{u \times m}$ of minimum distance at least $d$. By the Singleton bound for the rank metric we deduce that
	$\dim_{\F_q}((A\mC)(E)) \le m(u-d+1)$. 
\end{proof}

Lemma \ref{lem:bnd} implies the $d$-invariance of codes arising from vector rank metric codes. In the proof of the following proposition, we will need the following observation.

\begin{remark} \label{remreader}
	Let $\Gamma$ be a $\F_q$-basis of $\F_{q^m}$. Then for $x \in \F_{q^m}^n$ and all $\alpha \in \F_{q^m}$ with $\alpha \neq 0$ we have
	$\sigma(\Gamma(x))=\sigma(\Gamma(\alpha x))$.
\end{remark}

\begin{proposition}\label{cor:dsupps}
	Let $C$ be an $\F_{q^m}$-$[n,k,d]$ code. Then for all
	$\F_q$-bases $\Gamma$
	of $\fqm$ the code $\Gamma(C)$ is $d$-invariant with
	$$\mu(\Gamma(C),d)=q^m-1.$$
\end{proposition}
\begin{proof}
	Let $\mC:=\Gamma(C)$, and let $U$ be an $d$-support of $\mC$. Then by Lemma \ref{lem:bnd} we have
	$|\mC_=(U)| \le q^{m}-1$. Let $x \in C$ be a codeword such that $\sigma(\Gamma(x)) =U$. By Remark~\ref{remreader}, the set
	$$\{\Gamma(\alpha x) \ : \ \alpha \in \fqm \setminus \{0\}\} \subseteq \mC$$
	is made of matrices whose support is $U$. Moreover, it has cardinality $q^m-1$.
	Therefore we have $|\mC_=(U)| \ge q^{m}-1$, concluding the proof.
\end{proof}

We can now combine the results obtained so far as follows.

\begin{corollary}\label{th:fqmam}
	Let $C$ be an $\F_{q^m}$-$[n,k,d]$ code. Let $1 \le t<d$ be an integer, and assume that
	$$|\{1 \le i \le n-t \ : \  W_i(C^\perp) \neq 0\}| \le d-t.$$
	Let $d^\perp$ be the minimum distance of $C^\perp$. Then for every $\F_q$-basis of $\fqm$
	the $d$-supports of $\Gamma(C)$ and the $d^\perp$-supports of $\Gamma(C)^*$ form the blocks of a $t$-design over $\fq$.
\end{corollary}

\begin{proof}
	From Proposition \ref{cor:dsupps}, $\Gamma(C)$ is $d$-invariant. Denote by $\overline{\Gamma}$ the trace-dual of the basis $\Gamma$. Then by Lemma \ref{dualalb} we have $\Gamma(C)^*=\overline{\Gamma}(C^\perp)$. Therefore by Proposition~\ref{cor:dsupps} the code
	$\Gamma(C^\perp)$ is $d^\perp$-invariant as well.
	Again by Lemma \ref{dualalb}, the codes $C^\perp$ and $\Gamma(C)^*$ have the same
	weight distribution.
	The desired result now follows as an immediate consequence of Theorem \ref{th:AM}.
\end{proof}

In \cite[Theorem 5.3]{eejit}, the authors show that a $\fqm$-$[2d,d,d]$ code with no words of weight $d+1$ give rise to a Steiner system from the supports of its 
words of weight $d$; that result may also be seen to be a consequence of Corollary \ref{th:fqmam}.

\begin{example}
	Let $s$ be a positive integer and let $m=2s$. Let $\{\alpha_1,...\alpha_m\}$ be a basis of $\fqm$ over $\fq$.
	Let $C$ be the $\fqm$-$[m,m-2,2]$ vector rank metric code with parity check matrix
	$$H = \left[  \begin{array}{llll}
	\alpha_1 & \alpha_2 & \cdots & \alpha_m \\
	\alpha_1^{q^s} & \alpha_2^{q^s} & \cdots & \alpha_m^{q^s}
	\end{array}\right]. $$
	It can be checked that 
	$$W_0(C^\perp)=1, \quad W_s(C^\perp) = \frac{(q^{2s}-1)^2}{q^s-1},\quad W_{2s}(C^\perp) = q^{4s}-\frac{(q^{2s}-1)^2}{q^s-1}-1,$$ 
	and that $W_i(C^\perp)=0$ otherwise.
	In particular, $C^\perp$ has $\fq$-ranks $\{0,s,2s\}$ and so $C$ satisfies the hypothesis of Corollary \ref{th:fqmam} with $t=1$.
	That is, $C^\perp$ has exactly one weight $s$ in $\{1,...,2s-1\}$.
	The supports of the codewords of $C$ of rank $2$ form a $1$-design over $\fq$ and the words of rank $s$ in $C^\perp$
	form a $1$-$(m,s,1)$ subspace design, which is in fact a spread of $\fq^n$.
	
\end{example}

In the sequel, we will use the following characterization of 
MRD codes.

\begin{proposition} \label{projcriterion}
	Let $\mC$ be an $\fq$-$[n \times m, k,d]$ code, and assume $m \ge n$. The following are equivalent:
	\begin{enumerate}
		\item $\mC$ is MRD, \label{uu}
		\item $d_\rk(\mC)+d_\rk(\mC^*)=n+2$, \label{dd}
		\item the projection on the last $n-d+1$ rows $\mC \to \F_q^{(n-d+1) \times m}$  is surjective.
	\end{enumerate}
\end{proposition}

\begin{proof}
	The projection is injective, as $\mC$ has minimum distance $d$. Therefore it is surjective if and only if
	$\dim(\mC)=m(n-d+1)$, i.e., if and only if $\mC$ is MRD.
	Since a code is MRD if and only if its dual is also MRD \cite{del}, from the MRD Singleton bound it follows that (\ref{uu}) and (\ref{dd}) are equivalent.
\end{proof}

We will use the following result from \cite[Lemma 28]{alb1}, which is actually a special case of \cite[Lemma 48]{alblatt}.
The following rank metric specialization is more convenient for readers unfamiliar with lattices.

\begin{lemma}\label{lem:cu}
	Let $m \geq n$ and let $\mC$ be an $\fq$-$[n \times m, k,d]$ code. Let $U$ be a $u$-dimensional subspace of $\fq^n$. Then
	$$|\mC(U)| = q^{k-m(n-u)}|\mC^\perp(U)|. $$	
\end{lemma}

\begin{corollary}\label{cor:mrddsupps}
	Let $\mC$ be an $\fq$-$[n \times m, m(n-d+1),d]$ code and let $U$ be a $d$-dimensional subspace of $\fq^n$. Then
	$|\mC(U)| = q^m$. 
\end{corollary}

\begin{proof}
	Since $\mC$ is MRD by hypothesis, so is its dual code $\mC^*$, which has minimum distance $d^*=n-d+2$.
	Therefore, $\mC^*(U^\perp) = \{0\}$, since no non-zero word in $\mC^*$ has column space contained in an $(n-d)$-dimensional subspace
	of $\fq^n$. The result now follows directly from Lemma \ref{lem:cu}.
\end{proof}

The next result shows that if $m \ge n$ the minimum weight codewords of an MRD $\fqmn$-linear code always hold a trivial design, and that these are the only codes with this property. 
Note that the trivial design has as blocks the collection of all  $k$-dimensional subspaces of $\F_q^n$ and has parameters $$t\mbox{-}\left(n,k,\qbin{n-t}{k-t}{q}\right)_q.$$

\begin{proposition} \label{prop:MRDtrivial}
	Let $C$ be an $\F_{q^m}$-$[n,k,d]$ code, and let $\Gamma$ be any $\F_q$-basis of $\fqm$. Assume $m \ge n$. Then the following are equivalent.
	\begin{enumerate}
		\item $C$ is MRD, \label{uno}
		\item the words of rank $d$ in $\Gamma(C)$ hold a trivial design over $\fq$. \label{due}
	\end{enumerate}
\end{proposition}

\begin{proof}
	Suppose that $C$ is MRD. Then $\Gamma(C)$ is MRD as well and so by Corollary~\ref{cor:mrddsupps}, for all $U \le \F_q^n$ of dimension $d$ we have
	$|\Gamma(C)(U)|=q^m$.
	Since $C$ and $\Gamma(C)$ have minimum distance $d$, there exists $X \in \Gamma(C)$ such that $\sigma(X)=U$. Since $U$ was arbitrary, the words of rank $d$ in $\Gamma(C)$ hold the trivial design. This shows that (\ref{uno}) implies (\ref{due}).
	
	Let us show that (\ref{due}) implies (\ref{uno}).
	Suppose that $v \in \fqmn$ is a vector, and that $\Gamma$ and $\Gamma'$ are $\F_q$-bases of $\fqm$. Then it is easy to see that
	$\Gamma'(v)=\Gamma(v)  A$, where $A$ is an invertible matrix. Therefore it suffices to prove that (\ref{due}) implies (\ref{uno}) for a basis $\Gamma$ of our choice.
	
	We henceforth assume that $\Gamma$ is given by the powers of a primitive element $\alpha$ of $\F_{q^m}$, i.e.,
	$\Gamma=\{1,\alpha,\alpha^2,...,\alpha^{m-1}\}$. 
	
	For every $i \in \{d,...,n\}$ we fix $x_i \in C$ with $\sigma(\Gamma(x_i))=\langle e_1,...,e_{d-1},e_i\rangle$. Then for all $i \in \{d,...,n\}$ the projection 
	$\pi_{d-1}(\Gamma(x_i))$ is a matrix having only one non-zero row for all $i$, namely the $(i-d+1)$th row. 
	
	We now fix an arbitrary $i \in \{d,...,n\}$.
	In the sequel, we denote by $M_i$ the $i$th row of a matrix $M$.
	We clearly have
	$$(\Gamma(x_i)M)_i = \Gamma(x_i)_i M$$
	for every matrix $M$.
	Observe that multiplication by $\alpha$ corresponds to multiplication from the right by $M(\alpha)$, the companion matrix of the minimal polynomial of $\alpha$. More precisely, we have
	$\Gamma(\alpha x)=\Gamma(x)M(\alpha)$ for all $x \in \F_{q^m}^n$. 
	
	We will now show that the $m$ vectors 
	\begin{equation} \label{list}\Gamma(x_i)_i ,\ \Gamma(x_i)_iM(\alpha),...,\ \Gamma(x_i)_iM(\alpha)^{m-1} \ \in \F_q^m
	\end{equation} are $\F_q$-linearly independent. Towards a contradiction, suppose that there exist elements
	$\lambda_0,...,\lambda_{m-1} \in \F_q$, not all zero, such that
	$$\sum_{j=0}^{m-1} \lambda_j \Gamma(x_i)_i M(\alpha)^{j}=0.$$
	Then we have
	\begin{equation} \label{par}
	\Gamma(x_i)_i \left(\sum_{j=0}^{m-1} \lambda_j M(\alpha)^{j}\right)=0.
	\end{equation}
	The matrix in parentheses in (\ref{par}) is invertible by the properties of the companion matrix $M(\alpha)$. This implies 
	$\Gamma(x_i)_i=0$, a contradiction, and therefore the vectors in~(\ref{list}) are linearly independent.
	
	All of this shows that the projection $$\Gamma(\mC) \to \F_q^{(n-d+1) \times m}$$ on the last $n-d+1$ rows is surjective. Therefore $\Gamma(C)$ is MRD by Proposition \ref{projcriterion}, which in turn implies that
	$C$ is MRD by definition.
\end{proof}

\begin{remark}
	We can also see that \ref{uno}. implies \ref{due}. in Proposition~\ref{prop:MRDtrivial} by the following simple argument. If $C$ is MRD then its weight distribution is determined~\cite{del}, and in particular 
	$$W_d(C)=W_d(\Gamma(C))=\qbin{n}{d}{q}(q^m-1).$$ 
	From Proposition \ref{cor:dsupps}, $\mu(\Gamma(C),d)=q^m-1$, which is the number of words of $C$ that have the same $d$-support. But then the total number of $d$-supports of $C$ is $$\qbin{n}{d}{q}.$$
\end{remark}

\begin{remark}
	It was shown in \cite{AM} that a Hamming metric code is MDS if and only if 
	its minimum weight codewords hold a trivial design.
	This can be also seen using an argument very similar to that given in Proposition \ref{prop:MRDtrivial}, as we now show.
	If the codewords of weight $d$ in an $\fq$-$[n,k,d]$ code $C$
	hold a trivial design then, 
	in particular, $\{1,2,...,d-1,i\}$ is the support of a word of $C$ for each
	$i \in \{d,...,n\}$.
	If $C'$ is obtained from $C$ by puncturing the first $d-1$ coordinates of $C$ then $C'$ contains all the standard basis vectors of $\fq^{n-d+1}$. Therefore $C'$, and hence $C$, has dimension $n-d+1$. 
\end{remark}

As the next result shows, $\fqm$-linearity is crucial in the proof of Proposition \ref{prop:MRDtrivial}. In fact, the same result does not hold in general unless we assume that the code $\Gamma(C)$ arises from an $\fqm$-linear code, even if we consider the larger class of dually quasi-MRD codes.

\begin{proposition}
	Assume $m \ge n+1$, and let $1 \le d \le n-1$ be an integer. There exists a code
	$\mC \le \F_q^{n \times m}$ with the following properties:
	\begin{enumerate} \setlength\itemsep{0em}
		\item $\mC$ is not MRD,
		\item $\mC$ is not dually QMRD,
		\item the minimum weight codewords of $\mC$ hold the trivial $d$-design.
	\end{enumerate}
\end{proposition}
\begin{proof}
	Since $m \ge n+1$, there exists an MRD code $\mD \le \F_q^{n \times (m-1)}$ of minimum distance $d$.  For $M \in \mD$, let $z(M) \in \F_q^{n \times m}$ be the matrix obtained from $M$ by appending a zero column. Define the code 
	$\mC:=\{z(M) : M \in \mD\} \le \F_q^{n \times m}$.
	
	Since $M$ and $z(M)$ have the same support for all $M \in \mD$, the codes $\mD$ and $\mC$ have the same minimum distance. Moreover, the minimum weight codewords of $\mC$
	hold the trivial design. On the other hand, $\mC^*$ has minimum distance 1, and therefore $d_\rk(\mC)+d_\rk(\mC^*)=d+1 \le n-1+1=n$.
	In particular, $\mC$ is neither MRD (by Proposition~\ref{projcriterion}), nor dually QMRD.
\end{proof}


\section{Existence Results}
We now discuss the existence of codes that meet the hypotheses of the Assmus-Mattson Theorem for the rank metric.
We use results on the external distance and the covering radius of a  code to derive necessary conditions for a matrix code $\mC$ to satisfy the hypotheses of Theorem \ref{th:AM}, for $m$ sufficiently large. Using similar arguments, we also obtain results that apply to codes for the Hamming distance.

Recall that the \textbf{external distance} of an $\F_{q}$-$[n\times m,k,d]$ code $\mC$ is defined to be the integer
$$\tau(\mC) := |\{1 \le i \le n  \ : \ W_i(\mC^*) \neq 0 \}|.$$

The {\em external distance bound}~\cite[Theorem 27]{byrneravagnani} for an $\fq$-$[n \times m,k,d]$
code $\mC$ relates the covering radius of $\mC$ to the number of non-zero weights of $\mC^*$.
More precisely, we have:
$$\rho(\mC) \leq \tau(\mC).$$

We also have the following lower bound on the rank-metric covering radius;
see~\cite[Theorem 8.12]{byrravpart} for details.

\begin{theorem}\label{CRmax}
	Let $\mC$ be an $\fq$-$[n \times m,k,d]$ linear code and 
	let $1 \leq r\leq d$ be an integer satisfying 
	$(r-1)m \ge \log_q(4)+n^2/4$. Then
	$$\rho(\mC) \ge d-r+1.$$  
\end{theorem} 

The previous theorem implies that if $m \geq \log_q(4)+n^2/4$, then $\rho(\mC) \ge d-1$. Combining this with the external distance bound we obtain that the number of non-zero weights in the dual code $\mC^*$ is at least $d-1$.

Now suppose that $\mC$ is an $\fq$-$[n \times m,k,d]$ linear code satisfying the hypothesis of Theorem~\ref{CRmax}. For each $1\le t \le d-1$ define $X_t:=|\{ 1 \le i \le n-t : W_i(\mC^*) \neq 0\}| $ and 
$Y_t:=|\{n-t+1 \le i \le n : W_i(\mC^*) \neq 0 \}|$. Then $X_t+Y_t = \tau(\mC)$ and $Y_t \leq t$.
In order for the hypothesis of Theorem \ref{th:AM} to hold for $\mC$, we require that
$$d-1 - Y_t \le \tau(\mC)-Y_t = X_t \leq d-t.$$
In particular, $Y_t \in \{t,t-1\}$.
If $Y_t = t-1$, then we have $X_t \geq d-t$, and hence $X_t=d-t$ by hypothesis. This forces $\tau(\mC) = d-1$.
If $Y_t = t$, then $X_t \geq d-t-1$, so either $X_t= d-t-1$ and $\tau(\mC)=d-1$, or $X_t=d-t$ and $\tau(\mC)=d$.
Therefore we have shown the following result.

\begin{proposition}
	Suppose that $\mC$ is an $\fq$-$[n \times m,k,d]$ linear code satisfying the hypothesis of Theorem \ref{th:AM}. If
	$m \geq \log_q(4)+n^2/4$, then $\tau(\mC) \in \{d-1,d\}$.
\end{proposition}

In the case that $C$ is an $\fqm$-$[n,k,d]$ linear code, it has been shown in \cite[Proposition 8.4]{byrravpart} that for sufficiently large
$m$ one has $ \rho(C) = n-k$. Therefore
by the {external distance bound} we have:
\begin{equation}\label{eq:tau}
d-1 \le n-k = \rho(C) \le \tau(C) \in \{d-1,d\}.
\end{equation}
Thus 
$n-k \in \{d-1,d\}$ if $m \ge n$ is sufficiently large.

We remark that if $\tau(C)=d$, then we must have $n-k=d$. To see this, suppose that $n-k=d-1$. Then $C$ is MRD, so its dual code has $n-k$ weights yielding the contradiction $\tau(C) =n-k<d$.
Therefore, if $m \ge n$ is sufficiently large then from~(\ref{eq:tau}) either 
\begin{itemize}
	\item $\tau(C)=d-1=n-k$, in which case $C$ is MRD,
	\item $\tau(C)=\rho(C)=d = n-k$, in which case $C$ is not MRD.
\end{itemize}

The next result follows directly from (\ref{eq:tau}).
\begin{proposition}
	Let $C$ be an $\fqm$-$[n,k,d]$ linear code and let $1\leq t \leq d-1$. If $m \ge n$ is sufficiently large and if $C$ satisfies the hypothesis of Theorem \ref{th:AM}, then $d \in \{n-k,n-k+1\}$. Furthermore, if $d=n-k+1$ then $C$ holds the trivial design.
\end{proposition}


\begin{remark}
	In fact, a similar conclusion can be drawn for Hamming metric codes satisfying the classical Assmus-Mattson Theorem \cite[Theorem 4.2]{AM}. For sufficiently large $q$, the Hamming metric covering radius of an $\fq$-$[n,k,d]$ code is equal to~$n-k$; see~\cite[Theorem 8.3]{byrravpart}. Therefore, from the Hamming metric external distance bound~\cite[Chapter 6]{macws} we have that $n-k = \rho(\mC) \leq \tau(\mC)$. We summarize this as follows.
\end{remark}

\begin{proposition}
	Let $C$ be an $\fq$-$[n,k,d]$ linear code and let $1\leq t \leq d-1$. If $q$ is sufficiently large and if $C$ satisfies the hypothesis of the classical Assmus-Mattson Theorem, then $d \in \{n-k,n-k+1\}$. Furthermore, if $d=n-k+1$ then $C$ holds the trivial design.
\end{proposition}

In the Hamming metric, the canonical example of a code yielding a 2-design (the Fano plane) is the binary Hamming $[7,4,3]$ code. This can be easily seen using the classical Assmus-Mattson Theorem, observing that its dual code, the simplex code, has only one non-zero weight. 
In contrast to this, the next result shows that in the rank metric any code whose dual has only one non-zero weight must have minimum distance upper bounded by 2.

\begin{proposition}
	Let $\mC$ be an $\fq$-$[n\times m,k,d]$ matrix code. Suppose that $\mC^*$ has only one non-zero weight. Then
	$d\le 2$. 
\end{proposition}

\begin{proof}
	We can assume $m \ge n$ without loss of generality. Since $\mC^*$ has only one non-zero weight, we have
	$\tau(\mC) \le 1$. By the external distance bound \cite[Theorem 27]{byrneravagnani} we then have
	$\rho(\mC) \le 1$. Since $m,n \ge 2$ by assumption, the code $\mC$ cannot be the entire space, and therefore $\rho(\mC)=1$.
	As a consequence, the balls of radius 1 centered at the codewords of $\mC$ cover the entire space. In particular,
	$$|\mC| \cdot b(1) \ge q^{mn},$$
	where
	$b(1):=1+(q^n-1)(q^m-1)/(q-1)$ is the cardinality of a ball of radius 1 in the metric space $(\mat,d_\rk)$.
	By the Singleton bound (Proposition \ref{stbound}) we conclude that
	$$\frac{q^{mn}}{b(1)} \le |\mC| \le q^{m(n-d+1)}.$$
	Therefore
	$$d \le 1+\log_q(b(1))/m <3,$$
	where the latter inequality follows from the fact that $b(1)<q^{2m}$, as the reader can verify using elementary methods from Calculus. This implies $d \le 2$, as desired.
\end{proof}

\bibliographystyle{siamplain}
	
\end{document}